\documentclass[12pt]{amsart}

\usepackage{amsthm,amssymb,amsmath,amscd}
\usepackage{amsthm}
\usepackage{pgf, tikz}
\usepackage{graphicx,subcaption}
\usepackage{verbatim}
\usepackage{todonotes}

\theoremstyle{definition}
\newtheorem{definition}{Definition}

\theoremstyle{plain}
\newtheorem{theorem}[definition]{Theorem}
{Lemma}
{Theorem}
\newtheorem{lemma}[definition]{Lemma}

\newtheorem{claim}[definition]{Claim}
{Conjecture}

\theoremstyle{remark}

\newcommand{\pp}{P^3_3}

\title[The multipartite Ramsey number for the 3-path]
{The multipartite Ramsey number\\ for the 3-path of length three}

\author{Tomasz \L{}uczak}
\address{Adam Mickiewicz University\\
Faculty of Mathematics and Computer Science\\
Umultowska 87,
 61-614 Pozna\'n, Poland}
\email{\tt tomasz@amu.edu.pl}
\author{Joanna Polcyn}
\address{Adam Mickiewicz University\\
Faculty of Mathematics and Computer Science\\
Umultowska 87,
 61-614 Pozna\'n, Poland}
\email{\tt joaska@amu.edu.pl}
\thanks{The first author partially supported by NCN grant 2012/06/A/ST1/00261.}



\begin{document}

\maketitle

\begin{abstract}
We study the Ramsey number  for the 3-path of length three and $n$ colors
and show that $R(\pp;n)\le \lambda_0 n+7\sqrt{n}$, for some explicit constant $\lambda_0=1.97466\dots$.
\end{abstract}


\section{Introduction}

Let $\pp$ be the 3-uniform hypergraph with the set of vertices $\{a,b,c$, $d,e,f,g\}$ and the set of edges $\{\{a,b,c\},\{c,d,e\},\{e,f,g\}\}$. The Ramsey  number $R(\pp;n)$ is the smallest integer $N$ such that any coloring of the edges of the complete 3-uniform hypergraph $K_N^3$ 
on $N$ vertices 
with $n$ colors leads to a monochromatic copy of $\pp$. It is easy to see 
that $R(\pp;n)\ge n+6$ (see \cite{GR,J}) and it is believed that this lower bound is sharp, i.e. that 
$R(\pp;n)=n+6$. However, so far  this conjecture has been confirmed only for  $n\le 10$ (see \cite{J,JPR,P,PR}).
On the other hand, from the fact that for $N\ge 8$ each $\pp$-free 3-uniform hypergraph $H$ on $N$ vertices satisfies
\begin{equation}\label{ex}
|H|\le \binom {N-1}{2},
\end{equation}
(see \cite{FJS} and~\cite{JPRt}), it 
follows that     
\begin{equation*}\label{3n}
R(n;\pp) \le 3n.
\end{equation*}
In \cite{LP}  the authors of this note improved the above upper bound  to 
\begin{equation}\label{2n}
R(n;\pp) \le 2n +\sqrt{18n+1}+2.
\end{equation}
Our argument relied on the fact that for 2-graphs the analogous multicolored Ramsey number for a `usual' 2-path of length three is know to be $2n+O(1)$, where the small hidden 
constant $O(1)$ depends only on divisibility of $n$ by 3 (see \cite{I}). Thus, it seemed the method we used in \cite{LP} could not be applied directly to get an upper bound better than 
$(\lambda +o(1))n$, for $\lambda <2$. 

The main result of this note is to match our previous approach with later results from \cite{LPp} and get the following estimate for $R(\pp;n)$.

\begin{theorem}\label{main}
Let 
\begin{equation*}\label{eqf}
f(\gamma)=(\gamma^3-3\gamma^2+6\gamma - 6)^2-72\gamma (2-\gamma)(\gamma-1)^2.
\end{equation*}
and let 
$\lambda_0=1.97466\dots$ be the solution to the equation $f(\gamma)=0$, such that
$f(\gamma)> 0$ whenever  $\gamma\in (\lambda_0, 2)$.  
Then 
	$$
	R(\pp;n) \le\lambda_0 n+7\sqrt{n}.
	$$
\end{theorem}

%
%

\section{Proof of Theorem \ref{main}}

Our argument is based on the following decomposition lemma proved in \cite{LPp}. 
Before we state it we need some definitions. We call a 3-graph $H$ {\em quasi-bipartite} if one can partition its set of vertices into three sets: $X=\{x_1,x_2,\dots, x_s\}$, $Y=\{y_1,y_2,\dots, y_s\}$, and 
$Z=\{z_1,z_2,\dots, z_t\}$ in such a way that all the edges of $H$  can be written as 
$\{x_i,y_i,z_j\}$ for some $i=1,2,\dots, s$, and $j=1,2,\dots, t$. By a {\em star} with center $v$ we mean any 3-graph in which each edge contains $v$. Then the following holds.

\begin{lemma}
\label{podzial}
	For any $\pp$-free 3-graph $H$ there exists a partition of its set of vertices
	$V=V_R\cup V_T\cup V_S$, such that subhypergraphs of $H$ defined as $H_R=\{h\in H:h\cap V_R\neq \emptyset\}$, $H_T=H[V_T]$ and $H_S=H\setminus(H_R\cap H_T)=\{h\in H[V\setminus V_R]: h\cap V_S\neq\emptyset\}$ satisfy the following three conditions:
	\begin{enumerate}
		\item[(i)]  $|H_R|\le 6|V_R|$,   
		\item[(ii)] $H_T$ is quasi-bipartite and $|H_T|\le |V_T|^2/8$,  
		\item[(iii)]  $H_S$ is a family of disjoint stars such that centers of these stars are in $V_T$ whereas all other vertices are in $V_S$, and $|H_S|\le {|V_S|\choose 2}$.
	\end{enumerate}
\end{lemma}


The following lemma is a direct  consequence of the above result.

\begin{lemma}\label{bigstar}
Let $H$ be a  $\pp$-free 3-graph $H$ on $N\ge 95$ such that for some $s$, 
$(N+3)/2\le s\le N-46$, we have	
	$$
	|H|\ge {s-1\choose 2}+{N-s\choose 2},
	$$
	and let $H=H_R\cup H_T\cup H_S$ be a decomposition of $H$ as described in Lemma~\ref{podzial}.
Then $H_S$ contains a star on  at least $s$ vertices. 
\end{lemma}

\begin{proof}
	Let $V=V_R\cup V_T\cup V_S$ be a partition of the set of vertices $H$ given by Lemma \ref{podzial}. 
Note that $|V_S|\ge s-1$, since otherwise
\begin{align*}
|H|&\le 6|V_R|+\frac{|V_T|^2}8+{|V_S|\choose 2} 
\le{|V_S|\choose 2}+ \frac{(N-|V_S|)^2}8\\
	&\le{s-2\choose 2}+\frac{(N-s+2)^2}8
	<{s-1\choose 2}+{N-s\choose 2}.
	\end{align*}
	Recall that $H_S$ is a collection of disjoint stars. Suppose that the largest of these stars consists of at most $s-1 > N/2$ vertices. Then one can easily verify that the number of edges in $H_S$ is maximised if $H_S$ consists of  two stars on $s-1$ and $|V_S|-(s-1)+2$ vertices respectively.  
Consequently
		\begin{align*}
			|H|&\le 6|V_R|+\frac{|V_T|^2}{8}+{s-2\choose 2}+{|V_S|-s+2\choose 2}\\&\le {s-2\choose 2}+{N-s+1\choose 2}
<{s-1\choose 2}+{N-s\choose 2},
		\end{align*}
again contradicting the fact that $|H|\ge {s-1\choose 2}+{N-s\choose 2}$. 
Thus, $H_S$ contains a star on at least $s$ vertices.
\end{proof}

\begin{proof}[Proof of Theorem~\ref{main}]
We show that if for given integers $N$ and $n$ one can find a coloring of edges of $K^3_N$  by $n$ colors without monochromatic copies of $\pp$, then $\gamma=(N-7\sqrt n)/n<\lambda_0$
where $\lambda_0$ is defined in Theorem \ref{main}.
Some parts of our argument are quite technical and, since we aim to prove the statement for every $n$, we start with few remarks which makes our future computations a bit easier. 

Note that since $\lambda_0>1.97$, 
we may assume that   $\gamma>1.9$. 
Moreover, due to  (\ref{2n}), it is enough to consider   $\gamma < 2$. 
Finally, since $R(n; \pp)\le 3n$ we can restrict to the case $n\ge 41$ (and hence $N>122$) because otherwise $3n<1.9n+7\sqrt{n}$.

Thus, for  $n\ge 41$ and $1.9<\gamma<2$, let us consider a coloring of edges of $K^3_N$, $N=\gamma n+7\sqrt{n}$, by $n$ colors without monochromatic copies of~$\pp$, 
and let $H_i$  denote the $\pp$-free hypergraph 
generated by the $i$-th color.
 
 We say that the $i$th color is \emph{rich} if
 \begin{equation}\label{rich}
 |H_i|\ge {n+6\sqrt{n}-1\choose 2}+{N-n-6\sqrt{n}\choose 2}.
 \end{equation}

\begin{claim}\label{cl1}
At least $\beta n$ colors are rich, where 
$$\beta=\frac{\gamma^3-3\gamma^2+6\gamma -6}{6(\gamma-1)}\,.$$
\end{claim}
\begin{proof}
Due to technical calculations it will be convenient to show the statement by contradiction. 
 Denote the number of rich colors by $\beta n$ and assume that 
	\begin{equation}\label{bet}
	\beta<\frac{\gamma^3-3\gamma^2+6\gamma -6}{6(\gamma-1)}
	<\frac 13
	\end{equation}
	Since by (\ref{ex}), for each $i\in [n]$ we have $|H_i|\le {N-1\choose 2}$,  
$$
\binom {N}3< \beta n \binom {N-1}2+ n(1-\beta)\left(  {n+6\sqrt{n}-1\choose 2}+{N-n-6\sqrt{n}\choose 2} \right)\,.$$
Now substituting $N=\gamma n+7\sqrt n$ and putting all leading terms on the left hand side of the equation we arrive at
\begin{align*}
[(\gamma^3-3\gamma^2+6\gamma -6)&-\beta (6\gamma-6)]n^3<
[\beta(36\gamma-30)-(21\gamma^2-6\gamma -30)]n^{5/2}\\
&+[(\beta (42-6\gamma)-(150\gamma-3\gamma^2-105)]n^2\\
&-[6\beta+400-42\gamma]n^{3/2}-[2\gamma-153]n-14\sqrt{n}.
\end{align*}
But for $1.9<\gamma <2$ and $0\le \beta <1/3$ the right hand side of the above equation is smaller than
$-19n^{5/2}-157n^2-316n^{3/2}+150n-14\sqrt n$ which, in turn, is negative for all natural~$n$. Consequently,
\[
[(\gamma^3-3\gamma^2+6\gamma -6)-\beta (6\gamma-6)]n^3<0,
\]
and thus
\[
\beta>\frac{\gamma^3-3\gamma^2+6\gamma -6}{6(\gamma-1)},
\]
contradicting (\ref{bet}).
\end{proof}

Recall that each $H_i$ is $\pp$-free and so one can apply to it Lemma \ref{podzial} to get a decomposition of  $H_i$ into three graphs, $H^i_R\cup H^i_T\cup H^i_S$. 
For all $i\in [n]$ let us `uncolor' all the triples in $H^i_R$ and mark them `blank',
and set $\hat H_i=H^i_T\cup H^i_S$. 
Let $G_i$ be the graph whose edges are pairs which belong to at 
least two hyperedges of $\hat H_i$ and fewer than $6\sqrt{n}$ blank triples. 
Note that, because of the structure of $\hat H_i$, $G_i$ is a forest consisting of stars.

We say that an edge of $G_i$ is \emph{private} if 
it is not an edge of any other graph $G_j$, $j\neq i$, and  \emph{public}
otherwise. By  $e_i$ and $e'_i$ we denote the number of private and public edges of $G_i$, respectively.
\emph{The weight} $w(G_i)$ of $G_i$ is defined as 
$$w(G_i)=e_i+e'_i/2.$$
Since $G_i$ is a forest we have also
\begin{equation}\label{forest}
w(G_i)\le e_i+e_i'<N.
\end{equation}

Note that at most 
$$
\frac{3\sum_{i=1}^{n}|H^i_R|}{6\sqrt{n}}\le \frac{3\sum_{i=1}^{n}6N}{6\sqrt{n}}=3\sqrt{n}N=3\sqrt{n}(\gamma n+7\sqrt{n})<6n^{3/2}+21n
$$
pairs of $K^2_N$ belong to at least $6\sqrt{n}$ blank triples. 
Since by the pigeonhole principle all pairs which are contained in fewer than $6\sqrt{n}$ blank
triples are edges of at least one $G_i$, we have 
\begin{equation}\label{eq1}
\binom {N}2-6n^{3/2}-21n \le \sum_{i\in [n]} w(G_i).
\end{equation}

Let us make the following easy yet crucial observation.


\begin{claim}\label{cl2}
If a color $i$ is rich, 
then  $G_i$ contains more than $2w(G_i)-N$ private edges. 
All of them form a star. 
\end{claim}
\begin{proof}
Since $G_i$ is a forest we have $e_i+e'_i< N$. Thus,
$$w(G_i)=e_i+e'_i/2< e_i+N/2-e_i/2=(e_i+N)/2,$$ 
and so  $G_i$ contains more than $2w(G_i)-N$ private edges. Note 
also that, by Lemma~\ref{bigstar}, $H_S^i$ contains the unique largest star $F$ on at least  
$n+6\sqrt{n}$ vertices.  Let us denote the center of this star by $w$. Then each 
edge of $G_i$ which does not contain $w$ is clearly contained in fewer than $N-n-6\sqrt{n}$ hyperedges of $\hat H_i$ and so
belongs to  at least $n$ triples of $\bigcup_{j\neq i}\hat H_j$. Thus, 
by the pigeonhole principle, each such edge
 must be public. Consequently, all private edges must contain $w$ and form in  
$G_i$ a large star.
\end{proof}

Let $I$ denote the set of all rich colors. 
As an immediate corollary of Claim~\ref{cl2},  we get the following inequality

\begin{equation*}\label{eq3}
\sum_{i\in I} (2w(G_i)-N)<\sum_{i\in I}e_i\le\binom {|I|}2+|I|(N-|I|))
\end{equation*}
which leads to 
\begin{align*}\label{eqq3}
\sum_{i\in I} w(G_i)&\le N|I|-|I|^2/4-|I|/4\,. 
\end{align*}
Thus, using  (\ref{forest}) and (\ref{eq1}) we get
\begin{equation*}\label{eq4}
\begin{aligned}
\binom {N}2 -6n^{3/2}-21n&\le  \sum_{i\in [n]} w(G_i)=\sum_{i\in I}w(G_i)+\sum_{i\notin I}w(G_i)\\
&< N|I|-|I|^2/4 -|I|/4+\sum_{i\notin I}N\\
&= N|I|-|I|^2/4 -|I|/4+(n-|I|)N\\
&\le nN-|I|^2/4.
\end{aligned}
\end{equation*}
Hence, using the estimate for the size of $I$ given by  Claim \ref{cl1} we arrive at
\begin{equation*}\label{eq6}
\begin{aligned}
&\left(\frac{\gamma^3-3\gamma^2+6\gamma -6}{6(\gamma-1)}\right)^2\frac{n^2}2\le \frac{|I|^2}2 < 2nN-2\binom {N}2 +12n^{3/2}+42n.
\end{aligned}
\end{equation*}
Now substituting $N=\gamma n+7\sqrt n$ and putting all leading terms on the left hand side of the inequality results in the following formula 
\begin{equation*}
\left(\frac{(\gamma^3-3\gamma^2+6\gamma -6)^2}{72(\gamma-1)^2}-\gamma (2-\gamma)\right)n^2< (26-14 \gamma) n^{3/2}+(\gamma-7) n+7 \sqrt n.
\end{equation*}
But for $1.9<\gamma <2$ and $n\ge 2$ we have 
$$(26-14 \gamma) n^{3/2}+(\gamma-7) n+7 \sqrt n  <0,$$ and so
$$
\left(\frac{(\gamma^3-3\gamma^2+6\gamma -6)^2}{72(\gamma-1)^2}-\gamma (2-\gamma)\right)n^2<0.
$$
Consequently, 
\begin{equation*}\label{eq5}
(\gamma^3-3\gamma^2+6\gamma - 6)^2< 72\gamma (2-\gamma)(\gamma-1)^2,
\end{equation*}
which implies that $\gamma$ is smaller than $\lambda_0$ defined in Theorem \ref{main}.
\end{proof}

\bibliographystyle{plain}

\end{document}